\documentclass{amsart}
\usepackage[utf8]{inputenc}
\usepackage[T1]{fontenc}
\usepackage[english]{babel}
\usepackage{amsthm}
\usepackage{amssymb}
\usepackage{amsmath}
\usepackage{empheq}
\usepackage{graphicx}
\usepackage[export]{adjustbox}
\usepackage[all,knot]{xy}
\usepackage{tikz}
\usetikzlibrary{shapes.geometric,patterns,decorations.pathreplacing}
\usepackage{tikz-cd}
\usepackage{enumitem}
\usepackage{fullpage}

\theoremstyle{definition}
\newtheorem{defi}{Definition}[section]

\theoremstyle{plain}
\newtheorem{thm}[defi]{Theorem}

\newtheorem{lemma}[defi]{Lemma}

\newtheorem{prop}[defi]{Proposition}

\theoremstyle{remark}
\newtheorem{ex}[defi]{Example}

\newtheorem{rem}[defi]{Remark}

\newcommand{\scaleValue}{0.35}

\newcommand{\ba}[1]{\begin{equation*} \begin{array}{#1}}
\newcommand{\ea}{\end{array} \end{equation*}} 
\newcommand{\noi}{\noindent}

\newcommand{\mbf}{\mathbf}
\newcommand{\mbb}{\mathbb}
\newcommand{\mfk}{\mathfrak}
\newcommand{\mcl}{\mathcal}
\newcommand{\ra}{\rightarrow}
\newcommand{\lra}{\longrightarrow}

\newcommand{\lma}{\longmapsto}

\newcommand{\gra}{\mbf{Gra}}
\newcommand{\lieb}{\mbf{Lieb}}
\newcommand{\lie}{\mbf{Lie}}
\newcommand{\assB}{\mbf{Assb}}
\newcommand{\dFun}{\mcl{D}}
\newcommand{\oFun}{\mathcal{O}}
\newcommand{\dLieb}[1]{\dFun(\lieb_{#1})}
\newcommand{\dLiebC}[1]{\dFunC(\lieb_{#1})}
\newcommand{\edom}{\mbf{End}}
\newcommand{\free}[1]{\text{#1}}
\newcommand{\grt}{\mfk{grt}}
\newcommand{\GRT}{\mbf{GRT}}
\newcommand{\sgn}{\text{sgn}}
\newcommand{\Span}{\text{span}}
\newcommand{\hoLieb}{\mbf{HoLieb}}
\newcommand{\deform}{\mbf{Def}}
\newcommand{\fcgc}{\mbf{fcGC}}
\newcommand{\fgc}{\mbf{fGC}}
\newcommand{\GC}{\mbf{GC}}
\newcommand{\dfgc}{\mbf{dfGC}}
\newcommand{\dFunC}{\mcl{D}_c}
\newcommand{\defDLieb}{\deform(\lieb_{c,d} \stackrel{i}{\ra} \dLiebC{c,d})}
\newcommand{\defLieb}{\deform(\lieb_{c,d} \stackrel{id}{\ra} \lieb_{c,d})}
\newcommand{\sym}{\mathbf{Sym}}

\newcommand{\icg}{\includegraphics[scale=0.3,valign=c]}
\newcommand{\icgDec}{\includegraphics[scale=0.3]}

\title{Polydifferential Lie bialgebras and graph complexes}
\author{Vincent \textsc{Wolff}}
\address{Mathematics Research Unit, University of Luxembourg, Maison du Nombre, 6 Avenue de la Fonte,
 L-4364 Esch-sur-Alzette, Grand Duchy of Luxembourg }
\email{vincent.wolff@uni.lu}

\begin{document}

\begin{abstract}
    We study the deformation complex of a canonical morphism $i$ from the properad of (degree shifted) Lie bialgebras $\lieb_{c,d}$ to its  polydifferential version $\dLieb{c,d}$ and show that it is quasi-isomorphic to the oriented graph complex $\GC^{\text{or}}_{c+d+1}$, up to one rescaling class. As the latter complex is quasi-isomorphic to the original graph complex $\GC_{c+d}$, we conclude that the space of homotopy non-trivial infinitesimal deformations of the canonical map $i$ can be identified with the Grothendieck-Teichmüller Lie algebra $\grt$; moreover, every such an infinitesimal deformation extends to a genuine deformation of the canonical morphism $i$ from $\lieb_{c,d}$ to $\dFun(\lieb_{c,d})$. The full deformation complex is described with the help of a new graph complex of so called entangled graphs, whose suitable quotient complex is shown to contain the tensor product $H(\GC_c) \otimes H(\GC_d)$ of cohomologies of Kontsevich graph complexes $\GC_c \otimes \GC_d$.
\end{abstract}

\maketitle

\begin{section}{Introduction}

The theory of operads and props plays a very important role nowadays in many branches of mathematics and mathematical physics, including homotopy algebra, algebraic topology, Lie theory and deformation quantization theory. The two most famous deformation quantization theories, the quantizations of Poisson structures \cite{MR2062626} and of Lie bialgebras \cite{MR1183474,MR1403351}, can be usefully formulated in terms of morphisms of suitable operads/props to  polydifferential closures of other operads/props (see \cite{MR4179596} and \cite{MR4436207}). This interpretation is very helpful in the homotopy classification of the aforementioned deformation quantization theories. These polydifferential closures are constructed in turn with the help of the polydifferential endofunctor $\dFun$ in the category of props \cite{MR4179596}

\[
\dFun:\mbf{Props} \lra \mbf{Props},
\]

and its reduced version

\[
\oFun:\mbf{Props} \lra \mbf{Operads},
\]

which was introduced in an earlier paper \cite{merkulov2015props}. The key property of the functor $\dFun$ applied to any concrete prop $P$ is the following one: for any representation $P \ra \edom_V$ in a dg vector space $V$, there is an associated representation $\dFun(P) \ra \edom_{\odot V}$ in the symmetric tensor algebra $\odot V$ of $V$, given in terms of polydifferential (with respect to the standard product in $\odot V$) operators.

Applying the functor $\oFun$ to the (prop closure of the) operad $\lie_d$ of (degree shifted) Lie algebras we obtain the operad $\oFun(\lie_d)$ of polydifferential Lie algebras which comes equipped \cite{wolff2023} with a natural morphism of operads

\begin{equation}
\label{eqn:morphism_operads}
\lie_d \lra \oFun(\lie_d).
\end{equation}

The main result of op. cit. paper identifies the deformation complex of this morphism with the standard Kontsevich graph complex $\GC_d$.

In the present paper we apply the functor $\dFun$ to the (prop closure of the) properad $\lieb_{c,d}$ of (degree shifted) Lie bialgebras to obtain a prop $\dFun(\lieb_{c,d})$. This prop is very useful in the study of the deformation theory of Lie bialgebras \cite{MR1183474,MR1403351}, as one can interpret any universal deformation quantization of Lie bialgebras as a morphism of props \cite{MR4179596}

\[
\assB \lra \dLieb{1,1}
\]

satisfying certain non-triviality conditions; here $\assB$ stands for the prop of associative bialgebras. This interpretation of the deformation quantization theory was used in \cite{MR4179596} to classify all homotopy non-trivial universal quantizations in terms of the Kontsevich graph complex. The prop $\dLieb{c,d}$ contains naturally a properad $\dFunC(\lieb_{c,d})$ spanned by connected decorated graphs, called the properad of polydifferential Lie bialgebras. Similarly as in the Lie case, there is a natural morphism of properads

\begin{equation}
\label{eqn:morphism_properads}
i:\lieb_{c,d} \lra \dLiebC{c,d}    
\end{equation}

(see Lemma \ref{lemma:morphism_lieb_dlieb} for the explicit formulae). The main purpose of this paper is to compute the cohomology of the associated deformation complex $\defDLieb$ and to show that it is quasi-isomorphic to the deformation complex $\defLieb$. The latter complex has been studied in \cite{merkulov_deformation_liebialgebras,merkulov2015props}, where it was shown that there are isomorphisms of cohomology groups

\[
H^{\bullet +1}(\defLieb) = H^\bullet(\GC^{\text{or}}_{c+d+1})=H^\bullet(\GC_{c+d}),
\]

where $\GC^{\text{or}}_{c+d+1}$ is the oriented graph complex. Thus the deformation complexes of both maps $(\ref{eqn:morphism_operads})$ and $(\ref{eqn:morphism_properads})$ have the same cohomology but for different reasons: in the Lie case, the inclusion map $\deform(\lie_d \stackrel{id}{\ra}\lie_d)$ into $\deform(\lie_d \ra \oFun(\lie_d))$ is almost trivial, while in the second case all the cohomology comes from the corresponding inclusion.


There is a monomorphism of deformation complexes

\[
0 \lra \defLieb \lra \defDLieb.
\]

We prove in this paper that this monomorphism is a quasi-isomorphism by studying the quotient complex $\fcgc_{c,d}$ of so called \textit{entangled graphs} and computing its cohomology. It is worth noting that $\fcgc_{c,d}$ contains a subcomplex $\GC_{c,d}^1$ spanned by univalent entangled graphs and that its quotient $\GC_{c,d}^{\geq 2}$ has rich cohomology. We show in particular that there is an injection of cohomology groups

\[
H(\GC_{c}) \otimes H(\GC_d) \lra H(\GC_{c,d}^{\geq 2})
\]

for any $c$ and $d$. For $c=d=2$ the cohomology in degree $0$ contains thus $\grt_1 \otimes \grt_1$. Note that there is a similar monomorphism 

\[
\deform(\lie_d \ra \lie_d) \lra \deform(\lie_d \ra \oFun(\lie_d))
\]

in the case of Lie algebras. However, by contrast to the above case of $\lieb_{c,d}$, all the non-trivial cohomology of $\deform(\lie_d \ra \oFun(\lie_d))$ sits in the quotient complex as shown in \cite{wolff2023}. So both complexes are quasi-isomorphic to one and the same Kontsevich graph complex, but for different reasons.

\end{section}

\section*{Acknowledgments}

I am very grateful to my supervisor Sergei Merkulov for his support and many valuable discussions.

\begin{section}{Lie bialgebras}

We recall the definition of the properad $\lieb_{c,d}$ governing Lie bialgebras.

Consider the free properad $\text{E}$ spanned by the $\mbb{S}$-bimodule $E=\{E(m,n)\}$ given by

\[
E(m,n):=\left\{ \begin{array}{ll}
\sgn_2^{\otimes |d|} \otimes 1_2[d-1] = \Span \langle\begin{tikzpicture}[baseline={([yshift=-0.5ex]current bounding box.center)},scale=0.3,transform shape]
\node (v0) at (0,-1.25) {};
\node (v1) [circle,draw,fill] at (0,0) {};
\node (v2) at (-1,1) {\huge $1$};
\node (v3) at (1,1) {\huge $2$};
\draw (v0) edge (v1);
\draw (v1) edge (v2);
\draw (v1) edge (v3);
\end{tikzpicture} = (-1)^d \begin{tikzpicture}[baseline={([yshift=-0.5ex]current bounding box.center)},scale=0.3,transform shape]
\node (v0) at (0,-1.25) {};
\node (v1) [circle,draw,fill] at (0,0) {};
\node (v2) at (-1,1) {\huge $2$};
\node (v3) at (1,1) {\huge $1$};
\draw (v0) edge (v1);
\draw (v1) edge (v2);
\draw (v1) edge (v3);
\end{tikzpicture}
\rangle & \text{if m=2,n=1}, \\
1_2 \otimes \sgn_2^{\otimes |c|}[c-1] = \Span \langle \begin{tikzpicture}[baseline={([yshift=-0.5ex]current bounding box.center)},scale=0.3,transform shape]
\node (v0) at (0,1.25) {};
\node (v1) [circle,draw,fill] at (0,0) {};
\node (v2) at (-1,-1) {\huge $1$};
\node (v3) at (1,-1) {\huge $2$};
\draw (v0) edge (v1);
\draw (v1) edge (v2);
\draw (v1) edge (v3);
\end{tikzpicture}= (-1)^c \begin{tikzpicture}[baseline={([yshift=-0.5ex]current bounding box.center)},scale=0.3,transform shape]
\node (v0) at (0,1.25) {};
\node (v1) [circle,draw,fill] at (0,0) {};
\node (v2) at (-1,-1) {\huge $2$};
\node (v3) at (1,-1) {\huge $1$};
\draw (v0) edge (v1);
\draw (v1) edge (v2);
\draw (v1) edge (v3);
\end{tikzpicture} \rangle & \text{if m=1,n=2}, \\
0 & \text{otherwise,}
\end{array}\right.
\]

where

\[
\sgn_n^{\otimes |d|}= \left\{ \begin{array}{ll}
   \sgn_n  &  \text{if $d$ is odd,} \\
    1_n & \text{if $d$ is even,} 
\end{array} \right.
\]

where $\sgn_n$ denotes the one-dimensional sign representation and $1_n$ denotes the trivial representation of $\mbb{S}_n$. The properad $\lieb_{c,d}$ is defined to be the quotient $\free{E}/I$ where $I$ is generated by the following relations

\begin{equation*} \left\{ \begin{array}{l}\begin{tikzpicture}[baseline={([yshift=-0.5ex]current bounding box.center)},scale=0.35,transform shape]
\node (v0) at (0,1) {};
\node (v1) [circle,draw,fill] at (0,0) {};
\node (v2) [circle,draw,fill] at (-1,-1) {};
\node (v3) at (1,-1) {\huge $3$};
\node (v4) at (0,-2) {\huge $2$};
\node (v5) at (-2,-2) {\huge $1$};
\draw (v0) edge (v1);
\draw (v1) edge (v2);
\draw (v1) edge (v3);
\draw (v2) edge (v4);
\draw (v2) edge (v5);
\end{tikzpicture}
+
\begin{tikzpicture}[baseline={([yshift=-0.5ex]current bounding box.center)},scale=0.35,transform shape]
\node (v0) at (0,1) {};
\node (v1) [circle,draw,fill] at (0,0) {};
\node (v2) [circle,draw,fill] at (-1,-1) {};
\node (v3) at (1,-1) {\huge $1$};
\node (v4) at (0,-2) {\huge $3$};
\node (v5) at (-2,-2) {\huge $2$};
\draw (v0) edge (v1);
\draw (v1) edge (v2);
\draw (v1) edge (v3);
\draw (v2) edge (v4);
\draw (v2) edge (v5);
\end{tikzpicture}
+
\begin{tikzpicture}[baseline={([yshift=-0.5ex]current bounding box.center)},scale=0.35,transform shape]
\node (v0) at (0,1) {};
\node (v1) [circle,draw,fill] at (0,0) {};
\node (v2) [circle,draw,fill] at (-1,-1) {};
\node (v3) at (1,-1) {\huge $2$};
\node (v4) at (0,-2) {\huge $1$};
\node (v5) at (-2,-2) {\huge $3$};
\draw (v0) edge (v1);
\draw (v1) edge (v2);
\draw (v1) edge (v3);
\draw (v2) edge (v4);
\draw (v2) edge (v5);
\end{tikzpicture}
=0 \\[12pt]
\begin{tikzpicture}[baseline={([yshift=-0.5ex]current bounding box.center)},scale=0.35,transform shape]
\node (v0) at (0,-2) {};
\node (v1) [circle,draw,fill] at (0,-1) {};
\node (v2) [circle,draw,fill] at (-1,0) {};
\node (v3) at (1,0) {\huge $3$};
\node (v4) at (0,1) {\huge $2$};
\node (v5) at (-2,1) {\huge $1$};
\draw (v0) edge (v1);
\draw (v1) edge (v2);
\draw (v1) edge (v3);
\draw (v2) edge (v4);
\draw (v2) edge (v5);
\end{tikzpicture}
+
\begin{tikzpicture}[baseline={([yshift=-0.5ex]current bounding box.center)},scale=0.35,transform shape]
\node (v0) at (0,-2) {};
\node (v1) [circle,draw,fill] at (0,-1) {};
\node (v2) [circle,draw,fill] at (-1,0) {};
\node (v3) at (1,0) {\huge $1$};
\node (v4) at (0,1) {\huge $3$};
\node (v5) at (-2,1) {\huge $2$};
\draw (v0) edge (v1);
\draw (v1) edge (v2);
\draw (v1) edge (v3);
\draw (v2) edge (v4);
\draw (v2) edge (v5);
\end{tikzpicture}
+
\begin{tikzpicture}[baseline={([yshift=-0.5ex]current bounding box.center)},scale=0.35,transform shape]
\node (v0) at (0,-2) {};
\node (v1) [circle,draw,fill] at (0,-1) {};
\node (v2) [circle,draw,fill] at (-1,0) {};
\node (v3) at (1,0) {\huge $2$};
\node (v4) at (0,1) {\huge $1$};
\node (v5) at (-2,1) {\huge $3$};
\draw (v0) edge (v1);
\draw (v1) edge (v2);
\draw (v1) edge (v3);
\draw (v2) edge (v4);
\draw (v2) edge (v5);
\end{tikzpicture}
=0 \\[12pt]
\begin{tikzpicture}[baseline={([yshift=-0.5ex]current bounding box.center)},scale=0.35,transform shape]
\node (v0) [circle] at (0,3) {\huge $1$};
\node (v1) [circle] at (2,3) {\huge $2$};
\node (v2) [circle,draw,fill] at (1,2) {};
\node (v3) [circle,draw,fill] at (1,1) {};
\node (v4) [circle] at (0,0) {\huge $1$};
\node (v5) [circle] at (2,0) {\huge $2$};
\draw (v0) edge (v2);
\draw (v1) edge (v2);
\draw (v2) edge (v3);
\draw (v3) edge (v4);
\draw (v3) edge (v5);
\end{tikzpicture}
-
\begin{tikzpicture}[baseline={([yshift=-0.5ex]current bounding box.center)},scale=0.35,transform shape]
\node (v0) [circle] at (-1,2.5) {\huge $1$};
\node (v1) [circle,draw,fill] at (0,1) {};
\node (v2) [circle] at (0,0) {\huge $1$};
\node (v3) [circle,draw,fill] at (1,1.5) {};
\node (v4) [circle] at (1,3) {\huge $2$};
\node (v5) [circle] at (2,0) {\huge $2$};
\draw (v0) edge (v1);
\draw (v2) edge (v1);
\draw (v3) edge (v1);
\draw (v3) edge (v4);
\draw (v3) edge (v5);
\end{tikzpicture}
-(-1)^c \begin{tikzpicture}[baseline={([yshift=-0.5ex]current bounding box.center)},scale=0.35,transform shape]
\node (v0) [circle] at (-1,2.5) {\huge $1$};
\node (v1) [circle,draw,fill] at (0,1) {};
\node (v2) [circle] at (0,0) {\huge $2$};
\node (v3) [circle,draw,fill] at (1,1.5) {};
\node (v4) [circle] at (1,3) {\huge $2$};
\node (v5) [circle] at (2,0) {\huge $1$};
\draw (v0) edge (v1);
\draw (v2) edge (v1);
\draw (v3) edge (v1);
\draw (v3) edge (v4);
\draw (v3) edge (v5);
\end{tikzpicture}
-(-1)^{c+d}\begin{tikzpicture}[baseline={([yshift=-0.5ex]current bounding box.center)},scale=0.35,transform shape]
\node (v0) [circle] at (-1,2.5) {\huge $2$};
\node (v1) [circle,draw,fill] at (0,1) {};
\node (v2) [circle] at (0,0) {\huge $2$};
\node (v3) [circle,draw,fill] at (1,1.5) {};
\node (v4) [circle] at (1,3) {\huge $1$};
\node (v5) [circle] at (2,0) {\huge $1$};
\draw (v0) edge (v1);
\draw (v2) edge (v1);
\draw (v3) edge (v1);
\draw (v3) edge (v4);
\draw (v3) edge (v5);
\end{tikzpicture}
-(-1)^d\begin{tikzpicture}[baseline={([yshift=-0.5ex]current bounding box.center)},scale=0.35,transform shape]
\node (v0) [circle] at (-1,2.5) {\huge $2$};
\node (v1) [circle,draw,fill] at (0,1) {};
\node (v2) [circle] at (0,0) {\huge $1$};
\node (v3) [circle,draw,fill] at (1,1.5) {};
\node (v4) [circle] at (1,3) {\huge $1$};
\node (v5) [circle] at (2,0) {\huge $2$};
\draw (v0) edge (v1);
\draw (v2) edge (v1);
\draw (v3) edge (v1);
\draw (v3) edge (v4);
\draw (v3) edge (v5);
\end{tikzpicture}=0.
\end{array} \right.
\end{equation*}

Its minimal resolution $\hoLieb_{c,d}$ is generated by the $\mbb{S}$-bimodule $E=\{E(m,n)\}_{m,n \geq 1 ,m+n \geq 3}$ given by

\[
E(m,n)=\sgn_m^{\otimes|d|} \otimes \sgn_n^{\otimes|c|}[-(1+c(1-m)+d(1-n))]=\Span \langle\begin{tikzpicture}[baseline={([yshift=-0.5ex]current bounding box.center)},scale=0.5,transform shape]
\node (v0) [circle,draw,fill] at (0,0) {};
\node (v1) at (-1,1) {\huge $1$};
\node (v2) at (1,1) {\huge $m$};
\node (v3) at (-1,-1) {\huge $1$};
\node (v4) at (1,-1) {\huge $n$};
\draw (v0) edge (v1);
\draw (v0) edge (v2);
\draw (v0) edge (v3);
\draw (v0) edge (v4);
\node at (0,-0.7) {\huge $\cdots$};
\node at (0,0.7) {\huge $\cdots$};
\end{tikzpicture}\rangle.
\]

Its differential acts on the generators by

\[
\delta \begin{tikzpicture}[baseline={([yshift=-0.5ex]current bounding box.center)},scale=0.5,transform shape]
\node (v0) [circle,draw,fill] at (0,0) {};
\node (v1) at (-1,1) {\huge $1$};
\node (v2) at (1,1) {\huge $m$};
\node (v3) at (-1,-1) {\huge $1$};
\node (v4) at (1,-1) {\huge $n$};
\draw (v0) edge (v1);
\draw (v0) edge (v2);
\draw (v0) edge (v3);
\draw (v0) edge (v4);
\node at (0,-0.7) {\huge $\cdots$};
\node at (0,0.7) {\huge $\cdots$};
\end{tikzpicture}= \sum_{\substack{I_1 \sqcup I_2 = \{1,\cdots m\} \\ J_1 \sqcup J_2 = \{1,\cdots,n\} \\ |I_1|,|J_2| \geq 0 \\ |I_2|,|J_1| \geq 1}}{(-1)^{\sigma(I_1 \sqcup I_2)+|I_1||I_2|+|J_1||J_2|}\begin{tikzpicture}[baseline={([yshift=-0.5ex]current bounding box.center)},scale=0.5,transform shape]
\node (v0) [circle,draw,fill] at (0,0) {};
\node (v1) at (-1,-1) {};
\node (v2) at (1,-1) {};
\node (v3) at (-1,1) {};
\node (v4) at (1,1) {};
\draw (v0) edge (v1);
\draw (v0) edge (v2);
\draw (v0) edge (v3);
\draw (v0) edge (v4);
\draw [decoration={brace,amplitude=5pt},decorate] (v2) -- (v1);
\draw [decoration={brace,amplitude=5pt},decorate] (v3) -- (v4);
\node at (0,0.7) {\huge $\cdots$};
\node at (0,-0.7) {\huge $\cdots$};
\node at (0,1.75) {\huge $I_1$};
\node at (0,-1.75) {\huge $J_1$};
\node (v00) [circle,draw,fill] at (2.5,0.5) {};
\draw (v0) edge (v00);
\node (v01) at (1.5,-0.5) {};
\node (v02) at (3.5,-0.5) {};
\node (v03) at (1.5,1.5) {};
\node (v04) at (3.5,1.5) {};
\draw (v00) edge (v01);
\draw (v00) edge (v02);
\draw (v00) edge (v03);
\draw (v00) edge (v04);
\draw [decoration={brace,amplitude=5pt},decorate] (v02) -- (v01);
\draw [decoration={brace,amplitude=5pt},decorate] (v03) -- (v04);
\node at (2.5,-0.2) {\huge $\cdots$};
\node at (2.5,1.2) {\huge $\cdots$};
\node at (2.5,-1.25) {\huge $J_2$};
\node at (2.5,2.25) {\huge $I_2$};
\end{tikzpicture}},
\]

where $\sigma(I_1,I_2)$ denotes the sign of the shuffle $I_1 \sqcup I_2$.

\end{section}

\begin{section}{Reminder on graph complexes and deformation theory of Lie bialgebras}

We recall the different graph complexes which we will encounter in the following. Let $G_{v,e}$ be the set of directed graphs with $v$ vertices and $e$ edges and define the full graph complex by

\[
\fgc_d := \prod_{v \geq 1}\prod_{e \geq 0} (\mbb{K}\langle G_{v,e} \rangle \otimes \mbb{K}[-d]^{\otimes v} \otimes \mbb{K}[d-1]^{\otimes e}\otimes_{\mbb{S}_v \times (\mbb{S}_2^e \ltimes \mbb{S}_e)}\text{sgn}_d^{\otimes e})[d],
\]

where $\mbb{S}_v \times (\mbb{S}_2^e \ltimes \mbb{S}_e)$ acts by renumbering vertices, renumbering edges and flipping the direction of edges. The differential in $\fgc_d$ is (up to signs) given by splitting vertices. In the present paper we will mainly consider connected graphs and thus the subcomplex $\fcgc_d \subset \fgc_d$. It has been shown \cite{MR3348138} that there is a quasi-isomorphism $\fcgc_d \ra \GC_d$ where $\GC_d$ is spanned by graphs with all vertices being at least bivalent. In addition there is a subcomplex $\GC_d^{\geq 3} \subset \GC_d$ spanned by graphs with all vertices at least trivalent and we have \cite{MR3348138} $H(\GC_d)=\bigoplus_{\substack{p \geq 1 \\ p \equiv 2d+1 \mod 4}}{\mbb{K}[d-p]} \oplus H(\GC_d^{\geq 3})$. Furthermore, for $d=2$, the cohomology in degree zero is given by the Grothendieck-Teichmüller Lie algebra $\grt_1$.

There is a directed version $\dfgc_d$ of $\fgc_d$ where we keep direction on edges, both of which are quasi-isomorphic \cite{MR3348138}. We are interested in the subcomplex $\GC_d^{or} \subset \dfgc_d$ spanned by connected graphs with all vertices being at least bivalent and with no directed cycles. For example

\[
\begin{tikzpicture}[baseline={([yshift=-0.5ex]current bounding box.center)},scale=0.5,transform shape]
\node (v0) [circle,draw] at (0,0) {};
\node (v1) [circle,draw] at (-1,-1) {};
\node (v2) [circle,draw] at (1,-1) {};
\draw [-latex] (v0) edge (v1);
\draw [-latex] (v1) edge (v2);
\draw [-latex] (v0) edge (v2);
\end{tikzpicture} \in \GC^{or}_d, \hspace{4mm} \begin{tikzpicture}[baseline={([yshift=-0.5ex]current bounding box.center)},scale=0.5,transform shape]
\node (v0) [circle,draw] at (0,0) {};
\node (v1) [circle,draw] at (-1,-1) {};
\node (v2) [circle,draw] at (1,-1) {};
\draw [-latex] (v0) edge (v1);
\draw [-latex] (v1) edge (v2);
\draw [-latex] (v2) edge (v0);
\end{tikzpicture} \notin \GC^{or}_d
\]

The cohomology of $\GC_d^{or}$ has been computed \cite{MR3312446} and is given by the cohomology of the standard graph complex $\GC_{d-1}$. This complex controls deformations of Lie bialgebras \cite{merkulov_deformation_liebialgebras}. More precisely it was shown that there is an explicit quasi-isomorphism between the deformation complex $\defLieb$ and $\GC_{c+d+1}^{or}$ and thus to $\GC_{c+d}$. In particular, in the case $c+d=2$ corresponding to the usual Lie bialgebras, there is an action (up to homotopy) of the Grothendieck-Teichmüller group $\GRT_1$.

\end{section}

\begin{section}{Polydifferential functor on props}

In \cite{MR4179596} there is an endofunctor in the category of (augmented) props

\ba{rccc}
\dFun: & \mbf{Props} & \lra & \mbf{Props} \\
& P & \lma & \dFun(P),
\ea

with the property that for any representation $P \ra \edom_V$ of a prop $P$ in some dg vector space $V$ there is an associated representation $\dFun(P) \ra \edom_{\odot V}$ of the polydifferential prop $\dFun(P)$ in the symmetric tensor algebra $\odot V$, such that elements of $P$ acts as polydifferential operators. Our interest lies in the prop $\dFun({\lieb}_{c,d})$ (technically in a properad derived from said prop, see below) obtained by applying the above functor to the prop closure $\overline{\lieb}_{c,d}$ of the properad of Lie bialgebras. Elements of the prop closure are given by disjoint unions of elements in $\lieb_{c,d}$, e.g.

\[
\begin{tikzpicture}[baseline={([yshift=-0.5ex]current bounding box.center)},scale=0.35,transform shape]
\node (v0) [circle,draw,fill] at (0,0) {};
\node (v1) at (-1,1) {\huge $1$};
\node (v2) at (1,1) {\huge $2$};
\node (v3) at (0,-1) {\huge $1$};
\draw (v0) edge (v1);
\draw (v0) edge (v2);
\draw (v0) edge (v3);
\node (v00) [circle,draw,fill] at (3,0) {};
\node (v01) at (2,-1) {\huge $2$};
\node (v02) at (4,-1) {\huge $3$};
\node (v03) at (3,1) {\huge $3$};
\draw (v00) edge (v01);
\draw (v00) edge (v02);
\draw (v00) edge (v03);
\end{tikzpicture}.
\]

Given $p \in \overline{\lieb}_{c,d}$ and a partition of the input and output vertices respectively, we obtain an element of $\dLieb{c,d}$ by combining all input respectively output vertices in each partition class into a single vertex. For example

\[
\begin{tikzpicture}[baseline={([yshift=-0.5ex]current bounding box.center)},scale=0.35,transform shape]
\node (v0) [circle,draw,fill] at (0,0) {};
\node (v1) at (-1,1) {\huge $1$};
\node (v2) at (1,1) {\huge $2$};
\node (v3) at (0,-1) {\huge $1$};
\draw (v0) edge (v1);
\draw (v0) edge (v2);
\draw (v0) edge (v3);
\node (v00) [circle,draw,fill] at (3,0) {};
\node (v01) at (2,-1) {\huge $1$};
\node (v02) at (4,-1) {\huge $2$};
\node (v03) at (3,1) {\huge $2$};
\draw (v00) edge (v01);
\draw (v00) edge (v02);
\draw (v00) edge (v03);
\end{tikzpicture}
\lra
\begin{tikzpicture}[baseline={([yshift=-0.5ex]current bounding box.center)},scale=0.35,transform shape]
\node (v0) [circle,draw] at (-1,2) {\huge $1$};
\node (v1) [circle,draw,fill] at (0,1) {};
\node (v2) [circle,draw] at (0,0) {\huge $1$};
\node (v3) [circle,draw,fill] at (1,1) {};
\node (v4) [circle,draw] at (1,2) {\huge $2$};
\node (v5) [circle,draw] at (2,0) {\huge $2$};
\draw (v0) edge (v1);
\draw (v2) edge (v1);
\draw (v3) edge (v4);
\draw (v3) edge (v5);
\draw (v3) edge (v2);
\draw (v1) edge (v4);
\end{tikzpicture}.
\]

Given $\Gamma_1,\Gamma_2 \in \dLieb{c,d}$ with $\Gamma_1,\Gamma_2$ having the same number of input respectively output vertices, the vertical composition $\Gamma_1 \square_v \Gamma_2$ is defined as follows: We erase the input respectively output vertices of $\Gamma_1$ respectively $\Gamma_2$. Next for each label $k$ we sum over all attachments of the hanging output edges labelled $k$ to the hanging input edges labelled $k$ or to the output vertices of the connected component of the output vertex $k$ of $\Gamma_2$. Finally we sum over all attachments of the remaining input edges labelled $k$ to the input vertices of the connected of the output vertex $k$ in $\Gamma_1$.

\end{section}

\begin{section}{Properad of polydifferential Lie bialgebras}

Given the polydifferential prop $\dLieb{c,d}$ associated to the properad of Lie bialgebras, we can associate a properad $\dLiebC{c,d}$, called properad of polydifferential Lie bialgebras, whose elements are required to be connected as graphs. For example

\[
\begin{tikzpicture}[baseline={([yshift=-0.5ex]current bounding box.center)},scale=0.35,transform shape]
\node (v0) [circle,draw] at (-1,2) {\huge $1$};
\node (v1) [circle,draw,fill] at (0,1) {};
\node (v2) [circle,draw] at (0,0) {\huge $1$};
\node (v3) [circle,draw,fill] at (1,1) {};
\node (v4) [circle,draw] at (1,2) {\huge $2$};
\node (v5) [circle,draw] at (2,0) {\huge $2$};
\draw (v0) edge (v1);
\draw (v2) edge (v1);
\draw (v3) edge (v4);
\draw (v3) edge (v5);
\draw (v3) edge (v2);
\draw (v1) edge (v4);
\end{tikzpicture} \in \dLiebC{c,d}, \hspace{4 mm} \begin{tikzpicture}[baseline={([yshift=-0.5ex]current bounding box.center)},scale=0.35,transform shape]
\node (v0) [circle,draw,fill] at (0,0) {};
\node (v1) [circle,draw] at (-1,1) {\huge $1$};
\node (v2) [circle,draw] at (1,1) {\huge $2$};
\node (v3) [circle,draw] at (0,-1) {\huge $1$};
\draw (v0) edge (v1);
\draw (v0) edge (v2);
\draw (v0) edge (v3);
\node (v00) [circle,draw,fill] at (3,0) {};
\node (v01) [circle,draw] at (2,-1) {\huge $2$};
\node (v02) [circle,draw] at (4,-1) {\huge $3$};
\node (v03) [circle,draw] at (3,1) {\huge $3$};
\draw (v00) edge (v01);
\draw (v00) edge (v02);
\draw (v00) edge (v03);
\end{tikzpicture} \notin \dLiebC{c,d}.
\]

An element of $\dLiebC{c,d}$ consists of three types of vertices, the output and input vertices, which are at the top respectively at the bottom. We call the black vertices the \textit{internal vertices} and any edge between two internal vertices will be called an \textit{internal edge}. The graph induced by deleting the input and output vertices is in general not connected and each connected component will be called an \textit{irreducible connected component}.

For elements $\Gamma_1,\Gamma_2 \in \dLiebC{c,d}$ the properadic compositions $\Gamma_1 \, _{i_1,\cdots,i_k}\circ_{j_1,\cdots,j_k} \, \Gamma_2$ are defined as follows:
We first erase the  input vertices $i_1, \cdots i_k$ of $\Gamma_1$ and output vertices $j_1, \cdots, j_k$ of $\Gamma_2$. Then if the input vertex $i$ is paired with the output $j$ , we first sum over all possible attachments of the input vertices labelled $i$ and and output edges labelled $j$. Secondly we sum over all possible attachments of the remaining input edges labelled $i$ to the input vertices of $\Gamma_2$ and similarly for the remaining output edges labelled $j$.  

\begin{ex}
\begin{enumerate}[label=\alph*),before=\setlength{\baselineskip}{23mm}]
\item $\begin{tikzpicture}[baseline={([yshift=-0.5ex]current bounding box.center)},scale=0.35,transform shape]
\node (v2) [circle,draw] at (2,2) {\huge $1$};
\node (v3) [circle,draw,fill] at (2,1) {};
\node (v4) [circle,draw] at (1,0) {\huge $1$};
\node (v5) [circle,draw] at (3,0) {\huge $2$};
\draw (v2) edge (v3);
\draw (v3) edge (v4);
\draw (v3) edge (v5);
\end{tikzpicture} {_{1}\circ}_{2} \begin{tikzpicture}[baseline={([yshift=-0.5ex]current bounding box.center)},scale=0.35,transform shape]
\node (v0) [circle,draw] at (0,2) {\huge $1$};
\node (v1) [circle,draw] at (2,2) {\huge $2$};
\node (v2) [circle,draw,fill] at (1,1) {};
\node (v3) [circle,draw] at (1,0) {\huge $1$};
\draw (v0) edge (v2);
\draw (v1) edge (v2);
\draw (v2) edge (v3);
\end{tikzpicture}=\begin{tikzpicture}[baseline={([yshift=-0.5ex]current bounding box.center)},scale=0.35,transform shape]
\node (v0) [circle,draw] at (-1,2.5) {\huge $1$};
\node (v1) [circle,draw,fill] at (0,1) {};
\node (v2) [circle,draw] at (0,0) {\huge $1$};
\node (v3) [circle,draw,fill] at (1,1.5) {};
\node (v4) [circle,draw] at (1,3) {\huge $2$};
\node (v5) [circle,draw] at (2,0) {\huge $2$};
\draw (v0) edge (v1);
\draw (v2) edge (v1);
\draw (v3) edge (v1);
\draw (v3) edge (v4);
\draw (v3) edge (v5);
\end{tikzpicture}+\begin{tikzpicture}[baseline={([yshift=-0.5ex]current bounding box.center)},scale=0.35,transform shape]
\node (v0) [circle,draw] at (-1,2) {\huge $1$};
\node (v1) [circle,draw,fill] at (0,1) {};
\node (v2) [circle,draw] at (0,0) {\huge $1$};
\node (v3) [circle,draw,fill] at (1,1) {};
\node (v4) [circle,draw] at (1,2) {\huge $2$};
\node (v5) [circle,draw] at (2,0) {\huge $2$};
\draw (v0) edge (v1);
\draw (v2) edge (v1);
\draw (v3) edge (v4);
\draw (v3) edge (v5);
\draw (v3) edge (v2);
\draw (v1) edge (v4);
\end{tikzpicture}$.

\item $\begin{tikzpicture}[baseline={([yshift=-0.5ex]current bounding box.center)},scale=0.35,transform shape]
\node (v0) [circle,draw] at (0,2) {\huge $1$};
\node (v1) [circle,draw,fill] at (1,1) {};
\node (v2) [circle,draw] at (2,2) {\huge $2$};
\node (v3) [circle,draw] at (1,0) {\huge $1$};
\draw (v0) edge (v1);
\draw (v1) edge (v2);
\draw (v1) edge (v3);
\end{tikzpicture} {_{1}\circ_1} \begin{tikzpicture}[baseline={([yshift=-0.5ex]current bounding box.center)},scale=0.35,transform shape]
\node (v0) [circle,draw] at (0,0) {\huge $1$};
\node (v1) [circle,draw,fill] at (1,1) {};
\node (v2) [circle,draw] at (2,0) {\huge $2$};
\node (v3) [circle,draw] at (1,2) {\huge $1$};
\draw (v0) edge (v1);
\draw (v1) edge (v2);
\draw (v1) edge (v3);
\end{tikzpicture}=
\begin{tikzpicture}[baseline={([yshift=-0.5ex]current bounding box.center)},scale=0.35,transform shape]
\node (v0) [circle,draw] at (0,3) {\huge $1$};
\node (v1) [circle,draw] at (2,3) {\huge $2$};
\node (v2) [circle,draw,fill] at (1,2) {};
\node (v3) [circle,draw,fill] at (1,1) {};
\node (v4) [circle,draw] at (0,0) {\huge $1$};
\node (v5) [circle,draw] at (2,0) {\huge $2$};
\draw (v0) edge (v2);
\draw (v1) edge (v2);
\draw (v2) edge (v3);
\draw (v3) edge (v4);
\draw (v3) edge (v5);
\end{tikzpicture} +
\begin{tikzpicture}[baseline={([yshift=-0.5ex]current bounding box.center)},scale=0.35,transform shape]
\node (v0) [circle,draw] at (-1,2) {\huge $1$};
\node (v1) [circle,draw,fill] at (0,1) {};
\node (v2) [circle,draw] at (0,0) {\huge $1$};
\node (v3) [circle,draw,fill] at (1,1) {};
\node (v4) [circle,draw] at (1,2) {\huge $2$};
\node (v5) [circle,draw] at (2,0) {\huge $2$};
\draw (v0) edge (v1);
\draw (v2) edge (v1);
\draw (v3) edge (v0);
\draw (v3) edge (v5);
\draw (v3) edge (v2);
\draw (v1) edge (v4);
\end{tikzpicture} +
\begin{tikzpicture}[baseline={([yshift=-0.5ex]current bounding box.center)},scale=0.35,transform shape]
\node (v0) [circle,draw] at (-1,2) {\huge $1$};
\node (v1) [circle,draw,fill] at (0,1) {};
\node (v2) [circle,draw] at (0,0) {\huge $1$};
\node (v3) [circle,draw,fill] at (1,1) {};
\node (v4) [circle,draw] at (1,2) {\huge $2$};
\node (v5) [circle,draw] at (2,0) {\huge $2$};
\draw (v0) edge (v1);
\draw (v5) edge (v1);
\draw (v3) edge (v0);
\draw (v3) edge (v5);
\draw (v3) edge (v2);
\draw (v1) edge (v4);
\end{tikzpicture} +
\begin{tikzpicture}[baseline={([yshift=-0.5ex]current bounding box.center)},scale=0.35,transform shape]
\node (v0) [circle,draw] at (-1,2) {\huge $1$};
\node (v1) [circle,draw,fill] at (0,1) {};
\node (v2) [circle,draw] at (0,0) {\huge $1$};
\node (v3) [circle,draw,fill] at (1,1) {};
\node (v4) [circle,draw] at (1,2) {\huge $2$};
\node (v5) [circle,draw] at (2,0) {\huge $2$};
\draw (v0) edge (v1);
\draw (v2) edge (v1);
\draw (v3) edge (v0);
\draw (v3) edge (v5);
\draw (v3) edge (v2);
\draw (v1) edge (v4);
\end{tikzpicture} + 
\begin{tikzpicture}[baseline={([yshift=-0.5ex]current bounding box.center)},scale=0.35,transform shape]
\node (v0) [circle,draw] at (-1,2) {\huge $1$};
\node (v1) [circle,draw,fill] at (0,1) {};
\node (v2) [circle,draw] at (0,0) {\huge $1$};
\node (v3) [circle,draw,fill] at (1,1) {};
\node (v4) [circle,draw] at (1,2) {\huge $2$};
\node (v5) [circle,draw] at (2,0) {\huge $2$};
\draw (v0) edge (v1);
\draw (v5) edge (v1);
\draw (v3) edge (v4);
\draw (v3) edge (v5);
\draw (v3) edge (v2);
\draw (v1) edge (v4);
\end{tikzpicture}. $
\end{enumerate}
\end{ex}

We observe that the compositions in $\dLiebC{c,d}$ does not decrease the number of internal edges and thus we can consider the quotient properad $\dLiebC{c,d}/I$ where $I$ is the ideal generated by all elements with at least one internal edge. This quotient can be fully described using hairy graphs (see §\ref{section:entangled graphs}).

\begin{lemma}\label{lemma:morphism_lieb_dlieb}

There is a morphism of properads

\[
i:\lieb_{c,d} \ra \dLiebC{c,d}
\]

\noi given by

\begin{equation*}
\left\{ \begin{array}{cll}
\begin{tikzpicture}[baseline={([yshift=-0.5ex]current bounding box.center)},scale=\scaleValue,transform shape]
\node (v0) at (0,1.25) {\huge $1$};
\node (v1) [circle,draw,fill] at (0,0) {};
\node (v2) at (-1,-1) {\huge $1$};
\node (v3) at (1,-1) {\huge $2$};
\draw (v0) edge (v1);
\draw (v1) edge (v2);
\draw (v1) edge (v3); 
\end{tikzpicture} & \lma & \begin{tikzpicture}[baseline={([yshift=-0.5ex]current bounding box.center)},scale=\scaleValue,transform shape]
\node (v0) [circle,draw] at (0,1.25) {\huge $1$};
\node (v1) [circle,draw,fill] at (0,0) {};
\node (v2) [circle,draw] at (-1,-1) {\huge $1$};
\node (v3) [circle,draw] at (1,-1) {\huge $2$};
\draw (v0) edge (v1);
\draw (v1) edge (v2);
\draw (v1) edge (v3);
\end{tikzpicture} \\[0.8cm]
\begin{tikzpicture}[baseline={([yshift=-0.5ex]current bounding box.center)},scale=\scaleValue,transform shape]
\node (v0) at (0,-1.25) {\huge $1$};
\node (v1) [circle,draw,fill] at (0,0) {};
\node (v2) at (-1,1) {\huge $1$};
\node (v3) at (1,1) {\huge $2$};
\draw (v0) edge (v1);
\draw (v1) edge (v2);
\draw (v1) edge (v3);
\end{tikzpicture} & \lma &
\begin{tikzpicture}[baseline={([yshift=-0.5ex]current bounding box.center)},scale=\scaleValue,transform shape]
\node (v0) [circle,draw] at (0,-1.25) {\huge $1$};
\node (v1) [circle,draw,fill] at (0,0) {};
\node (v2) [circle,draw] at (-1,1) {\huge $1$};
\node (v3) [circle,draw] at (1,1) {\huge $2$};
\draw (v0) edge (v1);
\draw (v1) edge (v2);
\draw (v1) edge (v3);
\end{tikzpicture}.
\end{array} \right.
\end{equation*}

\end{lemma}
\begin{proof}

The proof of the Jacobi and co-Jacobi identity follows from Lemma 3.2 in \cite{wolff2023}. It remains to show that the  compatibility condition is satisfied in $\dLiebC{c,d}$:

\ba{rcl}
\begin{tikzpicture}[baseline={([yshift=-0.5ex]current bounding box.center)},scale=0.35,transform shape]
\node (v0) [circle,draw] at (0,3) {\huge $1$};
\node (v1) [circle,draw] at (2,3) {\huge $2$};
\node (v2) [circle,draw,fill] at (1,2) {};
\node (v3) [circle,draw,fill] at (1,1) {};
\node (v4) [circle,draw] at (0,0) {\huge $1$};
\node (v5) [circle,draw] at (2,0) {\huge $2$};
\draw (v0) edge (v2);
\draw (v1) edge (v2);
\draw (v2) edge (v3);
\draw (v3) edge (v4);
\draw (v3) edge (v5);
\end{tikzpicture} &+&
\begin{tikzpicture}[baseline={([yshift=-0.5ex]current bounding box.center)},scale=0.35,transform shape]
\node (v0) [circle,draw] at (-1,2) {\huge $1$};
\node (v1) [circle,draw,fill] at (0,1) {};
\node (v2) [circle,draw] at (0,0) {\huge $1$};
\node (v3) [circle,draw,fill] at (1,1) {};
\node (v4) [circle,draw] at (1,2) {\huge $2$};
\node (v5) [circle,draw] at (2,0) {\huge $2$};
\draw (v0) edge (v1);
\draw (v2) edge (v1);
\draw (v3) edge (v0);
\draw (v3) edge (v5);
\draw (v3) edge (v2);
\draw (v1) edge (v4);
\end{tikzpicture} +
\begin{tikzpicture}[baseline={([yshift=-0.5ex]current bounding box.center)},scale=0.35,transform shape]
\node (v0) [circle,draw] at (-1,2) {\huge $1$};
\node (v1) [circle,draw,fill] at (0,1) {};
\node (v2) [circle,draw] at (0,0) {\huge $1$};
\node (v3) [circle,draw,fill] at (1,1) {};
\node (v4) [circle,draw] at (1,2) {\huge $2$};
\node (v5) [circle,draw] at (2,0) {\huge $2$};
\draw (v0) edge (v1);
\draw (v5) edge (v1);
\draw (v3) edge (v0);
\draw (v3) edge (v5);
\draw (v3) edge (v2);
\draw (v1) edge (v4);
\end{tikzpicture} +
\begin{tikzpicture}[baseline={([yshift=-0.5ex]current bounding box.center)},scale=0.35,transform shape]
\node (v0) [circle,draw] at (-1,2) {\huge $1$};
\node (v1) [circle,draw,fill] at (0,1) {};
\node (v2) [circle,draw] at (0,0) {\huge $1$};
\node (v3) [circle,draw,fill] at (1,1) {};
\node (v4) [circle,draw] at (1,2) {\huge $2$};
\node (v5) [circle,draw] at (2,0) {\huge $2$};
\draw (v0) edge (v1);
\draw (v2) edge (v1);
\draw (v3) edge (v4);
\draw (v3) edge (v5);
\draw (v3) edge (v2);
\draw (v1) edge (v4);
\end{tikzpicture} + 
\begin{tikzpicture}[baseline={([yshift=-0.5ex]current bounding box.center)},scale=0.35,transform shape]
\node (v0) [circle,draw] at (-1,2) {\huge $1$};
\node (v1) [circle,draw,fill] at (0,1) {};
\node (v2) [circle,draw] at (0,0) {\huge $1$};
\node (v3) [circle,draw,fill] at (1,1) {};
\node (v4) [circle,draw] at (1,2) {\huge $2$};
\node (v5) [circle,draw] at (2,0) {\huge $2$};
\draw (v0) edge (v1);
\draw (v5) edge (v1);
\draw (v3) edge (v4);
\draw (v3) edge (v5);
\draw (v3) edge (v2);
\draw (v1) edge (v4);
\end{tikzpicture} \\ 
& - & \left( \begin{tikzpicture}[baseline={([yshift=-0.5ex]current bounding box.center)},scale=0.35,transform shape]
\node (v0) [circle,draw] at (-1,2.5) {\huge $1$};
\node (v1) [circle,draw,fill] at (0,1) {};
\node (v2) [circle,draw] at (0,0) {\huge $1$};
\node (v3) [circle,draw,fill] at (1,1.5) {};
\node (v4) [circle,draw] at (1,3) {\huge $2$};
\node (v5) [circle,draw] at (2,0) {\huge $2$};
\draw (v0) edge (v1);
\draw (v2) edge (v1);
\draw (v3) edge (v1);
\draw (v3) edge (v4);
\draw (v3) edge (v5);
\end{tikzpicture}+\begin{tikzpicture}[baseline={([yshift=-0.5ex]current bounding box.center)},scale=0.35,transform shape]
\node (v0) [circle,draw] at (-1,2) {\huge $1$};
\node (v1) [circle,draw,fill] at (0,1) {};
\node (v2) [circle,draw] at (0,0) {\huge $1$};
\node (v3) [circle,draw,fill] at (1,1) {};
\node (v4) [circle,draw] at (1,2) {\huge $2$};
\node (v5) [circle,draw] at (2,0) {\huge $2$};
\draw (v0) edge (v1);
\draw (v2) edge (v1);
\draw (v3) edge (v4);
\draw (v3) edge (v5);
\draw (v3) edge (v2);
\draw (v1) edge (v4);
\end{tikzpicture} \right)  \\[10mm] 
& -(-1)^c & \left( \begin{tikzpicture}[baseline={([yshift=-0.5ex]current bounding box.center)},scale=0.35,transform shape]
\node (v0) [circle,draw] at (-1,2.5) {\huge $1$};
\node (v1) [circle,draw,fill] at (0,1) {};
\node (v2) [circle,draw] at (0,0) {\huge $2$};
\node (v3) [circle,draw,fill] at (1,1.5) {};
\node (v4) [circle,draw] at (1,3) {\huge $2$};
\node (v5) [circle,draw] at (2,0) {\huge $1$};
\draw (v0) edge (v1);
\draw (v2) edge (v1);
\draw (v3) edge (v1);
\draw (v3) edge (v4);
\draw (v3) edge (v5);
\end{tikzpicture}+\begin{tikzpicture}[baseline={([yshift=-0.5ex]current bounding box.center)},scale=0.35,transform shape]
\node (v0) [circle,draw] at (-1,2) {\huge $1$};
\node (v1) [circle,draw,fill] at (0,1) {};
\node (v2) [circle,draw] at (0,0) {\huge $2$};
\node (v3) [circle,draw,fill] at (1,1) {};
\node (v4) [circle,draw] at (1,2) {\huge $2$};
\node (v5) [circle,draw] at (2,0) {\huge $1$};
\draw (v0) edge (v1);
\draw (v2) edge (v1);
\draw (v3) edge (v4);
\draw (v3) edge (v5);
\draw (v3) edge (v2);
\draw (v1) edge (v4);
\end{tikzpicture} \right) \\[10mm] 
& -(-1)^{c+d} & \left( \begin{tikzpicture}[baseline={([yshift=-0.5ex]current bounding box.center)},scale=0.35,transform shape]
\node (v0) [circle,draw] at (-1,2.5) {\huge $2$};
\node (v1) [circle,draw,fill] at (0,1) {};
\node (v2) [circle,draw] at (0,0) {\huge $2$};
\node (v3) [circle,draw,fill] at (1,1.5) {};
\node (v4) [circle,draw] at (1,3) {\huge $1$};
\node (v5) [circle,draw] at (2,0) {\huge $1$};
\draw (v0) edge (v1);
\draw (v2) edge (v1);
\draw (v3) edge (v1);
\draw (v3) edge (v4);
\draw (v3) edge (v5);
\end{tikzpicture}+\begin{tikzpicture}[baseline={([yshift=-0.5ex]current bounding box.center)},scale=0.35,transform shape]
\node (v0) [circle,draw] at (-1,2) {\huge $2$};
\node (v1) [circle,draw,fill] at (0,1) {};
\node (v2) [circle,draw] at (0,0) {\huge $2$};
\node (v3) [circle,draw,fill] at (1,1) {};
\node (v4) [circle,draw] at (1,2) {\huge $1$};
\node (v5) [circle,draw] at (2,0) {\huge $1$};
\draw (v0) edge (v1);
\draw (v2) edge (v1);
\draw (v3) edge (v4);
\draw (v3) edge (v5);
\draw (v3) edge (v2);
\draw (v1) edge (v4);
\end{tikzpicture} \right) \\[10mm] 
& -(-1)^d & \left( \begin{tikzpicture}[baseline={([yshift=-0.5ex]current bounding box.center)},scale=0.35,transform shape]
\node (v0) [circle,draw] at (-1,2.5) {\huge $1$};
\node (v1) [circle,draw,fill] at (0,1) {};
\node (v2) [circle,draw] at (0,0) {\huge $1$};
\node (v3) [circle,draw,fill] at (1,1.5) {};
\node (v4) [circle,draw] at (1,3) {\huge $2$};
\node (v5) [circle,draw] at (2,0) {\huge $2$};
\draw (v0) edge (v1);
\draw (v2) edge (v1);
\draw (v3) edge (v1);
\draw (v3) edge (v4);
\draw (v3) edge (v5);
\end{tikzpicture}+\begin{tikzpicture}[baseline={([yshift=-0.5ex]current bounding box.center)},scale=0.35,transform shape]
\node (v0) [circle,draw] at (-1,2) {\huge $2$};
\node (v1) [circle,draw,fill] at (0,1) {};
\node (v2) [circle,draw] at (0,0) {\huge $1$};
\node (v3) [circle,draw,fill] at (1,1) {};
\node (v4) [circle,draw] at (1,2) {\huge $1$};
\node (v5) [circle,draw] at (2,0) {\huge $2$};
\draw (v0) edge (v1);
\draw (v2) edge (v1);
\draw (v3) edge (v4);
\draw (v3) edge (v5);
\draw (v3) edge (v2);
\draw (v1) edge (v4);
\end{tikzpicture} \right).
\ea

This now follows from the observations 

\vspace{3mm}

\begin{itemize}
\setlength{\itemsep}{15pt}
    \item $\begin{tikzpicture}[baseline={([yshift=-0.5ex]current bounding box.center)},scale=0.35,transform shape]
\node (v0) [circle,draw] at (-1,2) {\huge $1$};
\node (v1) [circle,draw,fill] at (0,1) {};
\node (v2) [circle,draw] at (0,0) {\huge $1$};
\node (v3) [circle,draw,fill] at (1,1) {};
\node (v4) [circle,draw] at (1,2) {\huge $2$};
\node (v5) [circle,draw] at (2,0) {\huge $2$};
\draw (v0) edge (v1);
\draw (v2) edge (v1);
\draw (v3) edge (v0);
\draw (v3) edge (v5);
\draw (v3) edge (v2);
\draw (v1) edge (v4);
\end{tikzpicture} = (-1)^d\begin{tikzpicture}[baseline={([yshift=-0.5ex]current bounding box.center)},scale=0.35,transform shape]
\node (v0) [circle,draw] at (-1,2) {\huge $2$};
\node (v1) [circle,draw,fill] at (0,1) {};
\node (v2) [circle,draw] at (0,0) {\huge $1$};
\node (v3) [circle,draw,fill] at (1,1) {};
\node (v4) [circle,draw] at (1,2) {\huge $1$};
\node (v5) [circle,draw] at (2,0) {\huge $2$};
\draw (v0) edge (v1);
\draw (v2) edge (v1);
\draw (v3) edge (v4);
\draw (v3) edge (v5);
\draw (v3) edge (v2);
\draw (v1) edge (v4);
\end{tikzpicture}$
    \item $\begin{tikzpicture}[baseline={([yshift=-0.5ex]current bounding box.center)},scale=0.35,transform shape]
\node (v0) [circle,draw] at (-1,2) {\huge $1$};
\node (v1) [circle,draw,fill] at (0,1) {};
\node (v2) [circle,draw] at (0,0) {\huge $1$};
\node (v3) [circle,draw,fill] at (1,1) {};
\node (v4) [circle,draw] at (1,2) {\huge $2$};
\node (v5) [circle,draw] at (2,0) {\huge $2$};
\draw (v0) edge (v1);
\draw (v5) edge (v1);
\draw (v3) edge (v0);
\draw (v3) edge (v5);
\draw (v3) edge (v2);
\draw (v1) edge (v4);
\end{tikzpicture} =(-1)^{c+d}\begin{tikzpicture}[baseline={([yshift=-0.5ex]current bounding box.center)},scale=0.35,transform shape]
\node (v0) [circle,draw] at (-1,2) {\huge $2$};
\node (v1) [circle,draw,fill] at (0,1) {};
\node (v2) [circle,draw] at (0,0) {\huge $2$};
\node (v3) [circle,draw,fill] at (1,1) {};
\node (v4) [circle,draw] at (1,2) {\huge $1$};
\node (v5) [circle,draw] at (2,0) {\huge $1$};
\draw (v0) edge (v1);
\draw (v2) edge (v1);
\draw (v3) edge (v4);
\draw (v3) edge (v5);
\draw (v3) edge (v2);
\draw (v1) edge (v4);
\end{tikzpicture}$

\item $\begin{tikzpicture}[baseline={([yshift=-0.5ex]current bounding box.center)},scale=0.35,transform shape]
\node (v0) [circle,draw] at (-1,2) {\huge $1$};
\node (v1) [circle,draw,fill] at (0,1) {};
\node (v2) [circle,draw] at (0,0) {\huge $1$};
\node (v3) [circle,draw,fill] at (1,1) {};
\node (v4) [circle,draw] at (1,2) {\huge $2$};
\node (v5) [circle,draw] at (2,0) {\huge $2$};
\draw (v0) edge (v1);
\draw (v5) edge (v1);
\draw (v3) edge (v4);
\draw (v3) edge (v5);
\draw (v3) edge (v2);
\draw (v1) edge (v4);
\end{tikzpicture}=(-1)^c\begin{tikzpicture}[baseline={([yshift=-0.5ex]current bounding box.center)},scale=0.35,transform shape]
\node (v0) [circle,draw] at (-1,2) {\huge $1$};
\node (v1) [circle,draw,fill] at (0,1) {};
\node (v2) [circle,draw] at (0,0) {\huge $2$};
\node (v3) [circle,draw,fill] at (1,1) {};
\node (v4) [circle,draw] at (1,2) {\huge $2$};
\node (v5) [circle,draw] at (2,0) {\huge $1$};
\draw (v0) edge (v1);
\draw (v2) edge (v1);
\draw (v3) edge (v4);
\draw (v3) edge (v5);
\draw (v3) edge (v2);
\draw (v1) edge (v4);
\end{tikzpicture}$.
    
\end{itemize}

\end{proof}

The aim of this paper is to study the deformation complex of the above morphism.

\end{section}

\begin{section}{Complex of entangled graphs}
\label{section:entangled graphs}

\begin{defi}
   We define the \textit{properad of entangled graphs} $\gra_{c,d}$ as the properad whose elements are pairs of  graphs $(\Gamma_1,\Gamma_2)$ where $\Gamma_1$ has an orientation induced for even $c$ by an ordering of the undirected edges $e_1 \wedge \cdots \wedge e_{\#E(\Gamma_1)}$  up to an even permutation, while an odd permutation gives a mulitplication by $-1$.  For $c$ odd we assume that edges are directed up to a flip and multiplication by $-1$. Similarly for $\Gamma_2$. In addition we assume that $\Gamma_1$ has hairs labelled by the edges of $\Gamma_2$ and vice-versa.
   Given two pairs $(\Gamma_1,\Gamma_2)$ and $(\gamma_1,\gamma_2)$ and given vertices $i_1,\cdots,i_k \in V(\Gamma_1)$ and $j_1,\cdots,j_k \in V(\gamma_2)$ we define the properadic composition $(\Gamma_1,\Gamma_2) \, _{i_1,\cdots,i_k}\circ_{j_1,\cdots,j_k} \, (\gamma_1,\gamma_2)$ is defined to be the pair $(G_1,G_2)$ where $G_1$ is the graph obtained by erasing the vertices $i_1,\cdots,i_k$ of $\Gamma_1$ and summing over all possible attachments of the halfedges to the vertices of $\gamma_1$. Similarly $G_2$ is obtained by erasing the vertices $j_1,\cdots,j_k$ of $\gamma_2$ and summing over all possible attachments of the halfedges to the vertices of $\Gamma_2$.
\end{defi}

\begin{ex}

The properadic composition 

\[
\left(\icg{example_composition_1},\includegraphics[scale=0.3,valign=t]{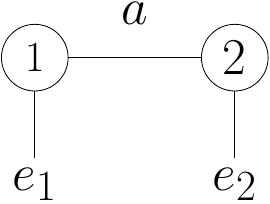}\right)   \, {_{1,2}\circ_{1,2}} \, \left(\includegraphics[scale=0.3,valign=t]{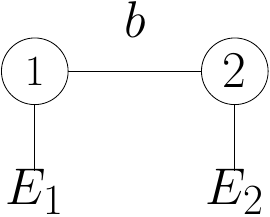},\icg{example_composition_3}\right)=\sum_{1 \leq i,j \leq 4}(\Gamma'_i,\Gamma_j'')
\]

with

\[
\Gamma'_1=\icg{compositions1_1},\hspace{2mm}\Gamma'_2=\icg{compositions1_2}, \hspace{2mm}\Gamma'_3=\icg{compositions1_3},\hspace{2mm} \Gamma'_4=\icg{compositions1_4},
\]

and 

\[
\Gamma''_1=\icg{compositions2_1},\hspace{2mm}\Gamma''_2=\icg{compositions2_2}, \hspace{2mm}\Gamma''_3=\icg{compositions2_3},\hspace{2mm} \Gamma''_4=\icg{compositions2_4}.
\]

\end{ex}

\begin{lemma}

Let $I \subset \dLiebC{c,d}$ be the ideal generated by the elements with at least one internal edge. There is an explicit isomorphism of properads

\[
\dLiebC{c,d}/I  \overset{F}{\lra} \gra_{c,d}.
\]

\end{lemma}
\begin{proof}

For a generator $G$

\[
\begin{tikzpicture}[baseline={([yshift=-0.5ex]current bounding box.center)},scale=0.35,transform shape]
\node (v0) [circle,draw] at (0,0) {\huge $1$};
\node (v1) [circle,draw] at (3,0) {\huge $n$};
\node (v2) [circle,draw] at (0,-3) {\huge $1$};
\node (v3) [circle,draw] at (3,-3) {\huge $m$};
\node (v4) at (1.5,0) {\huge $\cdots$};
\node (v5) at (1.5,-3) {\huge $\cdots$};
\node (v6) [circle,draw,fill] at (0.75,-1.5) {};
\draw (v0) edge (v6);
\draw (v6) edge (v4);
\draw (v6) edge (v2);
\node (v7) [circle,draw,fill] at (2.25,-1.5) {};
\draw (v2) edge (v7);
\draw (v7) edge (v3);
\draw (v7) edge (v1);
\end{tikzpicture}
\]

we associate a pair of graphs $(\Gamma_1,\Gamma_2)$ as follows: the vertices of $\Gamma_1$ respectively the vertices of $\Gamma_2$ correspond to the input respectively output vertices of $G$. There is an edge between two vertices of $\Gamma_1$ if the corresponding input vertices are connected by some tree of the form

\[
\begin{tikzpicture}[baseline={([yshift=-0.5ex]current bounding box.center)},scale=0.35,transform shape]
\node (V0) at (0,1.25)  {};
\node (v1) [circle,draw,fill] at (0,0) {};
\node (v2) at (-1,-1)  {};
\node (v3) at (1,-1) {};
\draw (V0) edge (v1);
\draw (v1) edge (v2);
\draw (v1) edge (v3);
\end{tikzpicture}.
\]

In addition each such tree points to exactly one output vertex of $\Gamma_2$ to whom we add an hair labelled by that edge.

Similarly the edges of $\Gamma_2$ (and hairs of $\Gamma_1$) are given by trees of the form

\[
\begin{tikzpicture}[baseline={([yshift=-0.5ex]current bounding box.center)},scale=0.35,transform shape]
\node (v0) at (0,-1.25) {};
\node (v1) [circle,draw,fill] at (0,0) {};
\node (v2) at (-1,1) {};
\node (v3) at (1,1) {};
\draw (v0) edge (v1);
\draw (v1) edge (v2);
\draw (v1) edge (v3); 
\end{tikzpicture}.
\]

The properadic compositions in $\gra_{c,d}$ are defined to make this map a morphism of properads. More precisely given $F(G_1)=(\Gamma_1,\Gamma_2)$ and given $F(G_2)=(\gamma_1,\gamma_2)$ we see that in the properadic composition $G_1$  $ _{i_1,\cdots,i_k}\circ_{j_1,\cdots,j_k} G_2$ deleting the input $i$ of $G_1$ and summing over all attachments to the input vertices of $G_2$ corresponds to deleting the vertex $i$ of $\Gamma_1$ and summing over all attachments of the halfedges of $i$ to the vertices of $\gamma_1$.

\end{proof}

Using the maps

\[
\lieb_{c,d} \ra \dLiebC{c,d} \ra \dLiebC{c,d}/I \ra \gra_{c,d}
\]

we define the \textit{complex of entangled graphs} by the deformation complex $\fcgc_{c,d}:=\deform(\lieb_{c,d} \ra \gra_{c,d} )$ of the above composition. Its elements are given by pairs of graphs $(\Gamma_1,\Gamma_2)$ with (skew)symmetrized vertices depending on the parity of $c$ respectively $d$. The differential $\delta$ is given by

\[
\delta(\Gamma_1,\Gamma_2)=\delta'(\Gamma_1,\Gamma_2) \pm \delta''(\Gamma_1,\Gamma_2),
\]

where

\[
\delta'(\Gamma_1,\Gamma_2)=\sum_{\stackrel{v\in V(\Gamma_1)}{w \in V(\Gamma_2)}}{(\begin{tikzpicture}[baseline={([yshift=-0.5ex]current bounding box.center)},scale=0.35,transform shape]
\node (v1) [circle,draw] at (0,0.5) {};
\node (v2) [circle,draw] at (0,-1) {\huge $\Gamma$};
\draw (v1) edge (v2);
\end{tikzpicture},\Gamma_2^w)}
\, \pm \, 
\sum_{\stackrel{v \in V(\Gamma_1)}{w\in V(\Gamma_2)}}{(\Gamma_1\circ_v \begin{tikzpicture}[baseline={([yshift=-0.5ex]current bounding box.center)},scale=0.35,transform shape]
\node (v1) at (0,0) [circle,draw] {};
\node (v2) at (1,0) [circle,draw] {};
\draw (v1) edge (v2);
\end{tikzpicture},\Gamma_2^w)},
\]

i.e. the first sum is given by attaching a univalent vertex with no hairs to any vertex of $\Gamma_1$, while the term $\Gamma_1 \circ_v \begin{tikzpicture}[baseline={([yshift=-0.5ex]current bounding box.center)},scale=0.35,transform shape]
\node (v1) at (0,0) [circle,draw] {};
\node (v2) at (1,0) [circle,draw] {};
\draw (v1) edge (v2);
\end{tikzpicture} $ is given by splitting the vertex $v$  and summing over all possible attachments of its halfedges (including hairs) to the newly created vertices. For a vertex $w$ of $\Gamma_2$ the graph $\Gamma_2^w$ is obtained by attaching an hair labelled by the new edge in $\Gamma_1$ to the vertex $w$. Similarly $\delta''(\Gamma_1,\Gamma_2)$ is given by switching the roles of $\Gamma_1$ and $\Gamma_2$.

\begin{rem}

If a vertex $v \in V(\Gamma_1)$ is of valency at least one then the terms adding a univalent vertex with no hairs are canceled with the similar terms created by splitting $v$. Note however that univalent vertices with hairs can be created by the differential. 

\end{rem}

The cohomology of $\fcgc_{c,d}$ is surprisingly simple as shown by the following theorem:

\begin{thm} \label{thm:cohomology_fcgc}
    The cohomology of $\fcgc_{c,d}$ is two-dimensional and generated by $(\begin{tikzpicture}[baseline={([yshift=-0.5ex]current bounding box.center)},scale=0.35,transform shape]
    \node (v0) [circle,draw] at (0,0) {};
    \node (v1) at (0,1) {};
    \draw (v0) edge (v1);
    \node at (0,-1) {};
    \end{tikzpicture},\begin{tikzpicture}[baseline={([yshift=-0.5ex]current bounding box.center)},scale=0.35,transform shape]
    \node (v0) [circle,draw] at (0,0) {};
    \node (v1) [circle,draw] at (1,0) {};
    \draw (v0) edge (v1);
    \end{tikzpicture})$ and $(\begin{tikzpicture}[baseline={([yshift=-0.5ex]current bounding box.center)},scale=0.35,transform shape]
    \node (v0) [circle,draw] at (0,0) {};
    \node (v1) [circle,draw] at (1,0) {};
    \draw (v0) edge (v1);
    \end{tikzpicture},\begin{tikzpicture}[baseline={([yshift=-0.5ex]current bounding box.center)},scale=0.35,transform shape]
    \node (v0) [circle,draw] at (0,0) {};
    \node (v1) at (0,1) {};
    \draw (v0) edge (v1);
    \node at (0,-1) {};
    \end{tikzpicture})$.
    
\end{thm}
\begin{proof}

Let $\fcgc_{c,d}^0$ be the differential closure of the subspace

\[
\text{span}\left\langle (\begin{tikzpicture}[baseline={([yshift=-0.5ex]current bounding box.center)},scale=0.35,transform shape]
    \node (v0) [circle,draw] at (0,0) {};
    \node (v1) at (0,1) {};
    \draw (v0) edge (v1);
    \node at (0,-1) {};
    \end{tikzpicture},\begin{tikzpicture}[baseline={([yshift=-0.5ex]current bounding box.center)},scale=0.35,transform shape]
    \node (v0) [circle,draw] at (0,0) {};
    \node (v1) [circle,draw] at (1,0) {};
    \draw (v0) edge (v1);
    \end{tikzpicture}), (\Gamma_1,\begin{tikzpicture}[baseline={([yshift=-0.5ex]current bounding box.center)},scale=0.35,transform shape]
    \node (v0) [circle,draw] at (0,0) {};
    \node (v1) at (0,1) {\huge $E(\Gamma_1)$};
    \draw (v0) [bend left] edge (v1);
    \draw (v0) [bend right] edge (v1);
    \node at (0,-1) {};
    \end{tikzpicture}) \, \text{for $\Gamma_1 \in \fcgc_c$} \right\rangle.
\]

Then the cohomology of $\fcgc_{c,d}^0$ is two-dimensional. Indeed we first observe that $\delta(\begin{tikzpicture}[baseline={([yshift=-0.5ex]current bounding box.center)},scale=0.35,transform shape]
    \node (v0) [circle,draw] at (0,0) {};
    \node (v1) at (0,1) {};
    \draw (v0) edge (v1);
    \node at (0,-1) {};
    \end{tikzpicture},\begin{tikzpicture}[baseline={([yshift=-0.5ex]current bounding box.center)},scale=0.35,transform shape]
    \node (v0) [circle,draw] at (0,0) {};
    \node (v1) [circle,draw] at (1,0) {};
    \draw (v0) edge (v1);
    \end{tikzpicture})=0$ and as by connectivity of the elements in $\fcgc_{c,d}$ there has to be always an edge it has to be a cohomology class. Similarly $(\begin{tikzpicture}[baseline={([yshift=-0.5ex]current bounding box.center)},scale=0.35,transform shape]
    \node (v0) [circle,draw] at (0,0) {};
    \node (v1) [circle,draw] at (1,0) {};
    \draw (v0) edge (v1);
    \end{tikzpicture},\begin{tikzpicture}[baseline={([yshift=-0.5ex]current bounding box.center)},scale=0.35,transform shape]
    \node (v0) [circle,draw] at (0,0) {};
    \node (v1) at (0,1) {};
    \draw (v0) edge (v1);
    \node at (0,-1) {};
    \end{tikzpicture})$ is a cohomology class. In addition if $\#E(\Gamma_1) \geq 2$ we have

    \[
    \delta (\Gamma_1,\begin{tikzpicture}[baseline={([yshift=-0.5ex]current bounding box.center)},scale=0.35,transform shape]
    \node (v0) [circle,draw] at (0,0) {};
    \node (v1) at (0,1) {\huge $E(\Gamma_1)$};
    \draw (v0) [bend left] edge (v1);
    \draw (v0) [bend right] edge (v1);
    \node at (0,-1) {};
    \end{tikzpicture}) =  (\delta\Gamma_1,\begin{tikzpicture}[baseline={([yshift=-0.5ex]current bounding box.center)},scale=0.35,transform shape]
    \node (v0) [circle,draw] at (0,0) {};
    \node (v1) at (0,1) {\huge $E(\delta\Gamma_1)$};
    \draw (v0) [bend left] edge (v1);
    \draw (v0) [bend right] edge (v1);
    \node at (0,-1) {};
    \end{tikzpicture})
    \pm \sum_{\stackrel{I \sqcup J=E(\Gamma_1)}{I,J \neq \emptyset}}{(\Gamma_1^e,\begin{tikzpicture}[baseline={([yshift=-0.5ex]current bounding box.center)},scale=0.35,transform shape]
    \node (v0) [circle,draw] at (0,0) {};
    \node (v1) [circle,draw] at (1,0) {};
    \node (v2) at (0,1) {\huge $I$};
    \node (v3) at (1,1) {\huge $J$};
    \draw (v0) edge (v1);
    \draw (v0) [bend left] edge (v2);
    \draw (v0) [bend right] edge (v2);
    \draw (v1) [bend left] edge (v3);
    \draw (v1) [bend right] edge (v3);
    \node at (0,-1) {};
    \end{tikzpicture})}.
    \]

The last sum is never zero and thus $(\Gamma_1,\begin{tikzpicture}[baseline={([yshift=-0.5ex]current bounding box.center)},scale=0.35,transform shape]
    \node (v0) [circle,draw] at (0,0) {};
    \node (v1) at (0,1) {\huge $E(\Gamma_1)$};
    \draw (v0) [bend left] edge (v1);
    \draw (v0) [bend right] edge (v1);
    \node at (0,-1) {};
    \end{tikzpicture})$ is never a cycle. 

The theorem is proven if we show that the inclusion $\fcgc_{c,d}^0 \hookrightarrow \fcgc_{c,d}$ is a quasi-isomorphism or equivalently  if we show that the quotient complex $\fcgc_{c,d}^{\geq 1}:=\fcgc_{c,d}/\fcgc_{c,d}^0$ is acyclic. We consider on $\fcgc_{c,d}^{\geq 1}$ a filtration on the number of vertices of $\Gamma_2$, i.e. for $p\geq 0$

\[
F_{-p}=\text{span} \left\langle (\Gamma_1,\Gamma_2) \text{with $\#V(\Gamma_2) \geq p$} \right\rangle.
\]

Given a pair $(\Gamma_1,\Gamma_2)$ in the associated graded complex $\mathbf{gr}\fcgc_{c,d}^{\geq 1}$ the induced differential acts by splitting vertices in $\Gamma_1$ and attaching hairs to $\Gamma_2$. In particular the vertices of edges of $\Gamma_2$ stay invariant and we can consider the larger complex $\widetilde{\mathbf{gr}\fcgc}_{c,d}^{\geq 1}$ spanned by pairs $(\Gamma_1,\Gamma_2)$ with the vertices and edges of $\Gamma_2$ being totally ordered. Denote by $w_{\min}$ respectively $e_{\min}$ the minimal vertex respectively minimal edge of $\Gamma_2$. As the differential increases the number of hairs in $\Gamma_2$ we can consider the filtration on $\widetilde{\mathbf{gr}\fcgc_{c,d}}^{\geq 1}$ by the total number of hairs attached to the to the vertices in the set $V(\Gamma_2) \setminus \{w_{\min}\}$. Let $\widehat{\mathbf{gr}\fcgc_{c,d}}^{\geq 1}$ be the associated graded complex. For a pair $(\Gamma_1,\Gamma_2)$ the induced differential acts as above on $\Gamma_1$ but the new hair is only attached to the vertex $w_{\min}$. Call a vertex $v_{\min} \in V(\Gamma_1)$ \textit{very special} if it is of the form

\[
\begin{tikzpicture}[baseline={([yshift=-0.5ex]current bounding box.center)},scale=0.5,transform shape]
\node (v0) [circle,draw,label=right: \huge  $v_{\min}$] at (0,0) {};
\node (v1) at (0,1.5) {\huge $e_{\min}$};
\node (v2) [circle,draw] at (0,-2)  {};
\draw (v0) edge (v1);
\draw (v2)  edge (v0);
\node at (0.5,-1) {\huge $E_s$};
\end{tikzpicture}
\]

and the edge $E_s$ is uniquely attached as a hair to $w_{\min}$. Note that $v_{\min}$ being very special is only defined in a pair $(\Gamma_1,\Gamma_2)$. Next call a maximal chain of bivalent vertices bold vertices of the form, for $k \geq 1$,

\[
\begin{tikzpicture}[baseline={([yshift=-0.5ex]current bounding box.center)},scale=0.5,transform shape]
\node (v0) at (0,1.5) {\huge $e_{\min}$};
\node (v1) [circle,draw,label=right:\huge $v_{\min}$] at (0,0) {};
\node (v2) [circle,draw] at (0,-2) {};
\node (v3) [circle,draw] at (0,-3) {};
\draw (v0) edge (v1);
\draw (v1) edge (v2);
\draw (v2) [dotted] edge (v3);
\node (v4) [draw] at (0,-5) {};
\draw (v3) edge (v4);
\node at (0.5,-1) {\huge $E_{s_1}$};
\node at (0.5,-4) {\huge $E_{s_k}$};
\end{tikzpicture},
\]

the minimal antenna if all edges $E_{s_1},E_{s_2},\cdots,E_{s_k}$ are uniquely attached as hairs to the vertex $w_{\min}$. The ending vertex of the minimal antenna is either of valency different from $2$, not bold or one of its edges is attached as a hair to a vertex different from $w_{\min}$. Call the vertices of the above \textit{special} and consider a filtration on $\widehat{\mathbf{gr}\fcgc_{c,d}}^{\geq 1}$ by the number of non-special vertices in pairs $(\Gamma_1,\Gamma_2)$. Let $\overline{\mathbf{gr}\fcgc_{c,d}}^{\geq 1}$ be the associated graded complex.
We finally show that $\overline{\mathbf{gr}\fcgc_{c,d}}^{\geq 1}$ is acyclic.

Let $C=\bigoplus_{k \geq 1} C_k$ be the subcomplex where $C_k$ is generated by the following pair

\[
(\Gamma:=\begin{tikzpicture}[baseline={([yshift=-0.5ex]current bounding box.center)},scale=0.5,transform shape]
\node (v0) at (0,1.5) {\huge $e_{\min}$};
\node (v1) [circle,draw,label=right:\huge $v_{\min}$] at (0,0) {};
\node (v2) [circle,draw] at (0,-2) {};
\node (v3) [circle,draw] at (0,-3) {};
\draw (v0) edge (v1);
\draw (v1) edge (v2);
\draw (v2) [dotted] edge (v3);
\node (v4) [draw] at (0,-5) {};
\draw (v3) edge (v4);
\node at (0.5,-1) {\huge $E_{s_1}$};
\node at (0.5,-4) {\huge $E_{s_k}$};
\end{tikzpicture},\begin{tikzpicture}[baseline={([yshift=-0.5ex]current bounding box.center)},scale=0.5,transform shape]
\node (v0) [circle,draw] at (0,0) {};
\node (v1) [circle,draw,label=below:\huge $w_{\min}$] at (1,0) {};
\draw (v0) edge (v1);
\node (v2) at (1,1) {\huge $E(\Gamma)$};
\draw (v1) [bend left] edge (v2);
\draw (v1) [bend right] edge (v2);
\node at (0,-1) {};
\end{tikzpicture}).
\]

Then $C$ is acyclic as the differential $\delta:C_k \ra C_{k+1}$ is injective for $k$ odd and zero for $k$ even. In addition the complex splits as $\overline{\mathbf{gr}\fcgc_{c,d}}^{\geq 1}=C \oplus C'$. For $k \geq 0$ denote by $C'_k$ the subspace of $C'$ spanned by pairs $(\Gamma_1,\Gamma_2)$ with the minimal antenna being of length $k$. The differential $\delta:C'_k \ra C'_{k+1}$ is  an injection for $k$ even and zero for $k$ odd (which was the opposite for the complex $C$). Hence the complex $C'$ is acyclic and the result is proven.

\end{proof}

As a small application of the above result, we will briefly discuss, what happens if we restrict ourselves to graphs with univalent vertices. More precisely, let $\GC_{c,d}^1$ be the subcomplex spanned by pairs $(\Gamma_1,\Gamma_2)$ with at least one univalent vertex. Then, contrary to the standard graph complex, the cohomology of $\GC_{c,d}^1$ is highly non-trivial. To see this, we first need to investigate the quotient complex $\GC_{c,d}^{\geq 2}:=\fcgc_{c,d}/\GC_{c,d}^1$. It is spanned by pairs $(\Gamma_1,\Gamma_2)$ with all vertices bivalent and an induced differential which does not create univalent vertices. The cohomology of $\GC_{c,d}^{\geq 2}$ is unknown at the moment, but we have a partial result:

\begin{prop}
    There is an inclusion of complexes

\ba{rccc}
\sym: &\GC_c \otimes \GC_d & \lra & \GC_{c,d}^{\geq 2} \\
& \Gamma_1 \otimes \Gamma_2 & \lma & \sum_{\stackrel{f_1:E(\Gamma_2) \ra V(\Gamma_1)}{f_2:E(\Gamma_1) \ra V(\Gamma_2)}}{(\Gamma_1^{f_1}),\Gamma_2^{f_2}}
\ea

i.e. we sum over all possible hair attachments. In addition the map induces an inclusion on the level of cohomology

\[
H(\GC_c) \otimes H(\GC_d) \ra H(\GC_{c,d}^{\geq 2}).
\]

\end{prop}
\begin{proof}

We need to show that the map $\sym$ respects the differentials. Indeed, while splitting a vertex in $\fcgc_{c,d}$, the only situation where we do not sum over all hair attachments, is when we create univalent vertices (as we do not create univalent bold vertices). As these terms disappear in $\GC_{c,d}^{\geq 2}$ we conclude that $\sym$ is a morphism of complexes.

In addition the complex $\GC_{c,d}^{\geq 2}$ splits as

\[
\GC_{c,d}^{\geq 2}=\sym(\GC_c \otimes \GC_d) \oplus C,
\]

(independent of the choice of a complement of $\sym(\GC_c \otimes \GC_d)$). This follows easily from the observation that if in a sum $\Gamma \in \GC_{c,d}^{\geq 2}$ there are hair attachments missing, then also in $d(\Gamma)$ there are some hair attachments missing and thus $d(\Gamma) \notin \sym(\GC_c \otimes \GC_d)$. This shows the inclusion in cohomology.

\end{proof}

The cohomology of $\GC_{c,d}^1$ can now be finally computed:

\begin{prop}

The cohomology of $\GC_{c,d}^1$ is equal to the cohomology of $\GC_{c,d}^{\geq 2}$ shifted by $1$, up to the additional classes $(\begin{tikzpicture}[baseline={([yshift=-0.5ex]current bounding box.center)},scale=0.35,transform shape]
    \node (v0) [circle,draw] at (0,0) {};
    \node (v1) at (0,1) {};
    \draw (v0) edge (v1);
    \node at (0,-1) {};
    \end{tikzpicture},\begin{tikzpicture}[baseline={([yshift=-0.5ex]current bounding box.center)},scale=0.35,transform shape]
    \node (v0) [circle,draw] at (0,0) {};
    \node (v1) [circle,draw] at (1,0) {};
    \draw (v0) edge (v1);
    \end{tikzpicture})$ and $(\begin{tikzpicture}[baseline={([yshift=-0.5ex]current bounding box.center)},scale=0.35,transform shape]
    \node (v0) [circle,draw] at (0,0) {};
    \node (v1) [circle,draw] at (1,0) {};
    \draw (v0) edge (v1);
    \end{tikzpicture},\begin{tikzpicture}[baseline={([yshift=-0.5ex]current bounding box.center)},scale=0.35,transform shape]
    \node (v0) [circle,draw] at (0,0) {};
    \node (v1) at (0,1) {};
    \draw (v0) edge (v1);
    \node at (0,-1) {};
    \end{tikzpicture})$. 
    
In particular there is an inclusion in cohomology

\[
H^\bullet(\GC_c \otimes \GC_d) \lra H^{\bullet+1}(\GC_{c,d}^1)
\]
    
\end{prop}
\begin{proof}

Define $\overline{\fcgc_{c,d}}=\fcgc_{c,d}\setminus\{(\begin{tikzpicture}[baseline={([yshift=-0.5ex]current bounding box.center)},scale=0.35,transform shape]
    \node (v0) [circle,draw] at (0,0) {};
    \node (v1) at (0,1) {};
    \draw (v0) edge (v1);
    \node at (0,-1) {};
    \end{tikzpicture},\begin{tikzpicture}[baseline={([yshift=-0.5ex]current bounding box.center)},scale=0.35,transform shape]
    \node (v0) [circle,draw] at (0,0) {};
    \node (v1) [circle,draw] at (1,0) {};
    \draw (v0) edge (v1);
    \end{tikzpicture}), (\begin{tikzpicture}[baseline={([yshift=-0.5ex]current bounding box.center)},scale=0.35,transform shape]
    \node (v0) [circle,draw] at (0,0) {};
    \node (v1) [circle,draw] at (1,0) {};
    \draw (v0) edge (v1);
    \end{tikzpicture},\begin{tikzpicture}[baseline={([yshift=-0.5ex]current bounding box.center)},scale=0.35,transform shape]
    \node (v0) [circle,draw] at (0,0) {};
    \node (v1) at (0,1) {};
    \draw (v0) edge (v1);
    \node at (0,-1) {};
    \end{tikzpicture})\}$ and $\overline{\GC_{c,d}^1}=\GC^1_{c,d}\setminus\{(\begin{tikzpicture}[baseline={([yshift=-0.5ex]current bounding box.center)},scale=0.35,transform shape]
    \node (v0) [circle,draw] at (0,0) {};
    \node (v1) at (0,1) {};
    \draw (v0) edge (v1);
    \node at (0,-1) {};
    \end{tikzpicture},\begin{tikzpicture}[baseline={([yshift=-0.5ex]current bounding box.center)},scale=0.35,transform shape]
    \node (v0) [circle,draw] at (0,0) {};
    \node (v1) [circle,draw] at (1,0) {};
    \draw (v0) edge (v1);
    \end{tikzpicture}), (\begin{tikzpicture}[baseline={([yshift=-0.5ex]current bounding box.center)},scale=0.35,transform shape]
    \node (v0) [circle,draw] at (0,0) {};
    \node (v1) [circle,draw] at (1,0) {};
    \draw (v0) edge (v1);
    \end{tikzpicture},\begin{tikzpicture}[baseline={([yshift=-0.5ex]current bounding box.center)},scale=0.35,transform shape]
    \node (v0) [circle,draw] at (0,0) {};
    \node (v1) at (0,1) {};
    \draw (v0) edge (v1);
    \node at (0,-1) {};
    \end{tikzpicture})\}$. The result is now proven using the short exact sequence

\[
0 \lra \overline{\GC_{c,d}^1} \lra \overline{\fcgc_{c,d}} \lra \overline{\fcgc_{c,d}}/\overline{\GC_{c,d}^1} \cong \GC_{c,d}^{\geq 2} \lra 0
\]

and the fact that $\overline{\fcgc_{c,d}}$ is acyclic by Theorem \ref{thm:cohomology_fcgc}. 

\end{proof}

\end{section}

\begin{section}{Main theorem}

Recall that we are interested in the deformation complex of the morphism defined in Lemma \ref{lemma:morphism_lieb_dlieb}. The main contribution is to compute its cohomology.

\begin{thm} \label{thm:main_theorem}
    \[
    H^\bullet(\defDLieb)=H^\bullet(\defLieb).
    \]
\end{thm}

Before we prove this, we will explain how to interpret elements of $\defDLieb$ as (hairy) graphs with three types of vertices. Given $\Gamma \in \defDLieb$ we first contract each irreducible connected component into a single vertex, called \textit{star} vertex, which is decorated by the corresponding element in $\defLieb$ together with a partition of the input and output vertices. Next we delete each star vertex decorated by \icg{Corolla}, add an edge between the two input vertices and add an hair decorated by said edge to the output vertex. Similarly for the star vertices decorated by \icg{coCorolla}.

\begin{ex}

\[
\begin{tikzpicture}[baseline={([yshift=-0.5ex]current bounding box.center)},scale=0.5,transform shape]
\node (v0) [circle,draw,fill] at (0,0) {};
\node (v1) [circle,draw,fill] at (1,1) {};
\node (v2) [circle,draw] at (-1,2) {};
\node (v3) [circle,draw] at (1,2) {};
\node (v4) [circle,draw] at (0,-1) {};
\node (v5) [circle,draw,fill] at (2,0) {};
\node (v6) [circle,draw] at (2,-1) {};
\draw (v0) edge (v1);
\draw (v0) edge (v2);
\draw (v1) [bend left]  edge (v3);
\draw (v1) [bend right]  edge (v3);
\draw (v0) edge (v4);
\draw (v4) edge (v5);
\draw (v6) edge (v5);
\draw (v5) [bend right] edge (v3);
\end{tikzpicture} \lra\begin{tikzpicture}[baseline={([yshift=-0.5ex]current bounding box.center)},scale=0.5,transform shape]
\node (v0) [star,star points=10,draw] at (0,0) {};
\node (v1) [circle,draw] at (-1,1) {};
\node (v2) [circle,draw] at (0,-1) {};
\node (v3) [circle,draw] at (0,-2) {};
\node (v4) [circle,draw] at (1,1) {};
\node (v5) at (1,2) {\huge $e$};
\draw (v0) edge (v1);
\draw (v0) edge (v2);
\draw (v0) [bend right] edge (v4);
\draw (v0) [bend left] edge (v4);
\draw (v2) edge (v3);
\draw (v4) edge (v5);
\node at (0.5,-1.5) {\huge $e$};
\node at (0.85,-0.5) {\icgDec{starVertexDecoration1}};
\end{tikzpicture}
\]

\[
\begin{tikzpicture}[baseline={([yshift=-0.5ex]current bounding box.center)},scale=0.5,transform shape]
\node (v0) [circle,draw,fill] at (0,0) {};
\node (v1) [circle,draw,fill] at (1,1) {};
\node (v2) [circle,draw] at (-1,2) {};
\node (v3) [circle,draw] at (1,2) {};
\node (v4) [circle,draw] at (0,-1) {};
\node (v5) [circle,draw,fill] at (2,0) {};
\node (v6) [circle,draw] at (2,-1) {};
\draw (v0) edge (v1);
\draw (v0) edge (v2);
\draw (v1)  edge (v3);
\draw (v1)  edge (v2);
\draw (v0) edge (v4);
\draw (v4) edge (v5);
\draw (v6) edge (v5);
\draw (v5) [bend right] edge (v3);
\end{tikzpicture} \lra \begin{tikzpicture}[baseline={([yshift=-0.5ex]current bounding box.center)},scale=0.5,transform shape]
\node (v0) [star,star points=10,draw] at (0,0) {};
\node (v1) [circle,draw] at (-1,1) {};
\node (v2) [circle,draw] at (0,-1) {};
\node (v3) [circle,draw] at (0,-2) {};
\node (v4) [circle,draw] at (1,1) {};
\node (v5) at (1,2) {\LARGE $e$};
\draw (v0) edge (v2);
\draw (v0) [bend left] edge (v1);
\draw (v0) [bend right] edge (v1);
\draw (v0) edge (v4);
\draw (v2) edge (v3);
\draw (v4) edge (v5);
\node at (0.5,-1.5) {\LARGE $e$};
\node at (0.85,-0.5) {\icgDec{starVertexDecoration2}};
\end{tikzpicture}
\]

\end{ex}

Note that there are no edges between two star vertices. Any edge attached to a star vertex will be called \textit{star} edge.

The differential $\delta$ of $\defDLieb$ can be written $\delta=\delta_{\text{Lieb}}+ \delta_{\text{Graph}}$ where the first part increases the number of internal edges (and creates star vertices) while the second part acts by splitting the input and output vertices as in $\fcgc_{c,d}$. 

\begin{proof} Theorem \ref{thm:main_theorem}

We can see $\defLieb$ in $\defDLieb$ as the subspace spanned by the elements with all input and output vertices being univalent, i.e.

\[
\defLieb \cong \text{span} \left\langle \begin{tikzpicture}[baseline={([yshift=-0.5ex]current bounding box.center)},scale=0.35,transform shape]
\node (v0) [circle,draw] at (0,1.25) {};
\node (v1) [circle,draw,fill] at (0,0) {};
\node (v2) [circle,draw] at (-1,-1) {};
\node (v3) [circle,draw] at (1,-1) {};
\draw (v0) edge (v1);
\draw (v1) edge (v2);
\draw (v1) edge (v3);
\end{tikzpicture},
\begin{tikzpicture}[baseline={([yshift=-0.5ex]current bounding box.center)},scale=0.35,transform shape]
\node (v0) [circle,draw] at (0,-1.25) {};
\node (v1) [circle,draw,fill] at (0,0) {};
\node (v2) [circle,draw] at (-1,1) {};
\node (v3) [circle,draw] at (1,1) {};
\draw (v0) edge (v1);
\draw (v1) edge (v2);
\draw (v1) edge (v3);
\end{tikzpicture},
\includegraphics[valign=c, scale=0.35]{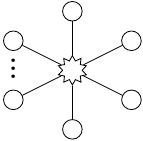} \right\rangle.
\]

As the graph part of the differential cannot create univalent bold vertices we have $\delta_{\text{Graph}} \equiv 0 $ on $\defLieb$ and there is an inclusion of complexes $\defLieb \hookrightarrow \defDLieb$. In addition the differential cannot reduce the number of vertices of valency at least $2$ and thus the complex $\defDLieb$ splits

\[
\defDLieb=\defLieb \oplus C,
\]

where $C$ the complex spanned by elements with at least one non-star vertex of valency greater than $1$. It remains to show that $C$ is acyclic. Consider a filtration on the number of internal edges. Then the induced differential on the associated graded complex $\mathbf{gr}C$ cannot create any new star vertices and thus the complex splits

\[
\mathbf{gr}C=\bigoplus_{N \geq 0}C_N,
\]

where $N$ is spanned by graphs having $N$ star vertices. Let $N \geq 1$. We show that $C_N$ is acyclic. Since the differential cannot create star vertices, the set of star edges is invariant and thus we can consider the complex $\widetilde{C}_N$ where the star edges are totally ordered. As there is at least one non-star vertex of valency greater than $1$ and by connectivity of the graphs, there is at least one star edge connected to a vertex of valency at least $2$. Denote by $e_{\min}$ the minimal such edge. We can split the complex $\widetilde{C}_{N}=\widetilde{C}_{N,in} \oplus \widetilde{C}_{N,out}$ where $\widetilde{C}_{N,in}$ respectively $\widetilde{C}_{N,out}$ is spanned by graphs where $e_{\min}$ is connected to an input respectively output vertex. We will show that $\widetilde{C}_{N,in}$ is acyclic, the proof for $\widetilde{C}_{N,out}$ is obtained by changing the roles of input and output vertices.

Consider a filtration on $\widetilde{C}_{N,in}$ by the number of output vertices. By construction the differential on $\mathbf{gr}\widetilde{C}_{N,in}$ acts only by splitting input vertices and attaches hairs to output vertices. In particular the set of output vertices is invariant and thus we can consider the complex $\widehat{\mathbf{gr}{C}}_{N,in}$ where we assume a total ordering on the set of output vertices. Denote by $w_{\min}$ the minimal output vertex and define a minimal antenna as a maximal chain of bivalent bold input vertices as below ($k \geq 1)$

\[
\begin{tikzpicture}[baseline={([yshift=-0.5ex]current bounding box.center)},scale=0.5,transform shape]
\node (v0) [star,star points=10,draw] at (0,0) {};
\node (v1) [circle,draw] at (2,0) {};
\draw (v0) edge (v1);
\node at (1,0.3) {\huge $e_{\min}$};
\node (v2) at (4,0) [circle,draw] {};
\draw (v1) edge (v2);
\node (v3) at (3,0.3) {\huge $E_{s_1}$};
\node (v4) [circle,draw] at (5,0) {};
\draw (v2) [dotted] edge (v4);
\node (v5)  [draw,fill] at (7,0) {};
\node at (6,0.3) {\huge $E_{s_k}$};
\draw (v4) edge (v5);
\end{tikzpicture},
\]

where we assume that all edges $E_{s_1},E_{s_2},\cdots,E_{s_{k}}$ are attached uniquely as hairs to the vertex $w_{\min}$. The square vertex is either a star vertex (and in that case we do not have a condition on $E_{s_k}$), an input vertex of valency different from $2$ or a vertex of valency $2$ which is either not bold or one of its edges is not uniquely attached as hair to $w_{\min}$. Call the vertices in the above antenna special and consider a filtration on the number of non-special vertices. Then the associated graded complex $\overline{\mathbf{gr}{C}}_{N,in}=\bigoplus_{k \geq 1}C'_k$ splits by the length of the minimal antenna. The differential $\delta_{\text{Graph}}:C'_k \ra C'_{k+1}$ is an injection if $k$ is odd and zero for $k$ even. Hence we conclude that $\overline{\mathbf{gr}{C}}_{N,in}$ is acyclic.

It remains to show that $C_0$ is acyclic. As the elements of $C_0$ do not have star vertices, they are given by two disjoint hairy graphs and thus the complex is isomorphic to $\fcgc_{c,d}\setminus\{(\begin{tikzpicture}[baseline={([yshift=-0.5ex]current bounding box.center)},scale=0.35,transform shape]
    \node (v0) [circle,draw] at (0,0) {};
    \node (v1) at (0,1) {};
    \draw (v0) edge (v1);
    \node at (0,-1) {};
    \end{tikzpicture},\begin{tikzpicture}[baseline={([yshift=-0.5ex]current bounding box.center)},scale=0.35,transform shape]
    \node (v0) [circle,draw] at (0,0) {};
    \node (v1) [circle,draw] at (1,0) {};
    \draw (v0) edge (v1);
    \end{tikzpicture}),(\begin{tikzpicture}[baseline={([yshift=-0.5ex]current bounding box.center)},scale=0.35,transform shape]
    \node (v0) [circle,draw] at (0,0) {};
    \node (v1) [circle,draw] at (1,0) {};
    \draw (v0) edge (v1);
    \end{tikzpicture},\begin{tikzpicture}[baseline={([yshift=-0.5ex]current bounding box.center)},scale=0.35,transform shape]
    \node (v0) [circle,draw] at (0,0) {};
    \node (v1) at (0,1) {};
    \draw (v0) edge (v1);
    \node at (0,-1) {};
    \end{tikzpicture})\}$, as these elements correspond to \icg{Corolla} and \icg{coCorolla} which are elements in $\defLieb$. By Theorem \ref{thm:cohomology_fcgc} we conclude that $C_0$ is acyclic as we removed all cohomology classes. This finishes the proof of the theorem.

\end{proof}

\begin{rem}
    The above result implies that any homotopy non-trivial deformation $\gamma$ of $i$ comes from the infinitesimal deformations of the identity map $id:\lieb_{c,d}\ra \lieb_{c,d}$. It has been proven in \cite{merkulov_deformation_liebialgebras} that every infinitesimal deformation of $id$ exponentiates to a genuine deformation. As $\dFun$ is a functor, the latter result implies that every such $\gamma$ exponentiates to a genuine deformation of $i$.
\end{rem}

\end{section}

\bibliographystyle{alpha}
\bibliography{bib}

\end{document}